\documentclass[a4paper,reqno,12pt,psamsfonts]{amsart}

\usepackage[headings]{fullpage}

\usepackage{upref}

\usepackage{amsthm}
\usepackage{amsmath}
\usepackage{amssymb}

\usepackage[all]{xy}
\usepackage{graphicx}

\usepackage{hyperref}

\theoremstyle{plain}
\newtheorem{thm}{Theorem}[section]
\newtheorem*{thm*}{Theorem}
\newtheorem{prop}[thm]{Proposition}
\newtheorem{lem}[thm]{Lemma}
\newtheorem{cor}[thm]{Corollary}

\theoremstyle{definition}
\newtheorem{defn}[thm]{Definition}

\theoremstyle{remark}
\newtheorem*{rem*}{Remark}
\newtheorem*{notation*}{Notation}
\newtheorem*{ex*}{Example}
\newtheorem*{ackn*}{Acknowledgements}

\theoremstyle{plain}
\newtheorem*{claim*}{Claim}

\newcommand{\Vol}{\operatorname{Vol}}
\newcommand{\diam}{\operatorname{diam}}
\newcommand{\FillRad}{\operatorname{FillRad}}
\newcommand{\FillVol}{\operatorname{FillVol}}
\newcommand{\dil}{\operatorname{dil}}
\newcommand{\sys}{\operatorname{sys}}
\newcommand{\SR}{\operatorname{SR}}
\newcommand{\FR}{\operatorname{FR}}
\newcommand{\FV}{\operatorname{FV}}
\newcommand{\F}{\operatorname{F}}

\begin{document}

\title[Filling inequalities do not depend on topology]{Filling inequalities do not depend on topology}
\author{Michael Brunnbauer}
\address{Mathematisches Institut, Ludwig-Maximilians-Universit\"at M\"unchen, Theresienstr.\ 39, D-80333 M\"unchen, Germany}
\email{michael.brunnbauer@mathematik.uni-muenchen.de}
\date{\today}

\begin{abstract}
Gromov's universal filling inequalities relate the filling radius and the filling volume of a Riemannian manifold to its volume. The main result of the present article is that in dimensions at least three the optimal constants in the filling inequalities depend only on dimension and orientability, not on the manifold itself. This contrasts with the analogous situation for the optimal systolic inequality, which does depend on the manifold.
\end{abstract}

\maketitle



\section{Introduction}

One of the most important curvature-free bounds on the volume of a Riemannian manifold is provided by Gromov's systolic inequality
\begin{equation}\label{sys_ineq}
\sys_1(M,g)^n \leq C_n\cdot \Vol(M,g), 
\end{equation}
which holds for all essential connected closed $n$-dimensional Riemannian manifolds $(M,g)$. Here, $M$ is called \emph{essential} if there exists an aspherical complex $K$ and a map $M\to K$ that does not contract to the $(n-1)$-skeleton of $K$. So in particular, all aspherical manifolds are essential. The \emph{systole} $\sys_1(M,g)$ is the length of the shortest noncontractible loop in $M$. A proof of \eqref{sys_ineq} can be found in \cite{Gromov(1983)}, Appendix 2, (B'$_1$). (Note also the paper \cite{Guth(2006)}, which contains a more detailed version of Gromov's proof.)

If one takes into account the topology of $M$, the optimal (smallest) value of the constant $C(M)$ such that an inequality \eqref{sys_ineq} holds for all Riemannian metrics $g$ on $M$ can be improved. This best value is given by the \emph{optimal systolic ratio}
\[ \SR(M) := \sup_g \frac{\sys_1(M,g)^n}{\Vol(M,g)}. \]

Its exact value is known only for three essential manifolds (apart from the trivial case of the circle): the two-torus, the real projective plane, and the Klein bottle. (See the survey article \cite{CK(2003)} and the book \cite{Katz(2007)}, which provide an extensive overview of systolic geometry, including a discussion of the filling invariants.) 

Nevertheless, it is known that in higher dimensions $\SR(M)$ also varies with $M$. For example $\SR(T^n)\geq 1$ as a trivial computation for a flat torus shows. But there exist hyperbolic manifolds $M_{hyp}$ with arbitrarily small optimal systolic ratio: Gromov proved an upper bound for the optimal systolic ratio by the simplicial volume
\[ \SR(M) \leq c'_n\frac{\log^n(1+\|M\|)}{\|M\|}, \]
which holds for all connected closed oriented manifolds (\cite{Gromov(1983)}, Theorem 6.4.D'). For hyperbolic manifolds the simplicial volume is known to be proportional to the volume:
\[ \|M_{hyp}\| = R_n^{-1} \cdot \Vol(M_{hyp}). \]
(This is due to Thurston, see \cite{Gromov(1982)}, page 11.) Since there are hyperbolic manifolds of arbitrarily large volume (just take a sequence of finite coverings with the number of sheets tending towards infinity), the above inequality shows that $\SR(M_{hyp})$ can become arbitrarily small.

To prove the systolic inequality \eqref{sys_ineq} Gromov introduced the \emph{filling radius} and the \emph{filling volume} of a Riemannian manifold. In \cite{Gromov(1983)}, Theorem 1.2.A and Theorem 2.3 he derived the filling inequalities
\begin{align*}
\FillRad(M,g)^n &\leq A_n \cdot\Vol(M,g) \;\;\;\mbox{and} \\
\FillVol(M,g)^{n/(n+1)} &\leq B_n \cdot\Vol(M,g),
\end{align*}
that hold for all connected closed Riemannian manifolds. (Recently, Wenger found a shorter proof for the second inequality, see \cite{Wenger(2007)}.) Again fixing the manifold $M$, one defines in analogy to the systolic ratio two optimal filling ratios:
\begin{align*}
\FR(M) &:= \sup_g \frac{\FillRad(M,g)^n}{\Vol(M,g)} \;\;\;\mbox{and} \\
\FV(M) &:= \sup_g \frac{\FillVol(M,g)^{n/(n+1)}}{\Vol(M,g)}.
\end{align*}
These topological invariants are the best values of the constants $A(M)$ and $B(M)$ such that the filling inequalities are true for all Riemannian metrics $g$ on $M$.

In contrast to the behaviour of the optimal systolic ratio, the main result of this article is the following:

\begin{thm}\label{main_thm_intro}
If $M$ and $N$ are connected closed manifolds of the same dimension $n\geq 3$ and if they are either both orientable or both nonorientable, then
\[ \FR(M) = \FR(N) \;\;\;\mbox{and}\;\;\; \FV(M) = \FV(N). \]
\end{thm}

This will be proved as a part of the more detailed Theorem \ref{main_thm}. The proof uses an axiomatic approach, that was first introduced by Babenko in \cite{Babenko(2006)} for the optimal systolic ratio and applied to further invariants in \cite{Sabourau(2006)} and \cite{Brunnbauer(2007a)}. (Implicitly, it can already be found in works of Babenko, Katz, and Suciu on systolic freedom. The keywords are `topological meromorphic maps' and `$(n,k)$-morphisms', see \cite{CK(2003)}, section 4.3.) 

The main idea is the following: both optimal filling ratios fulfill a `comparison axiom' that roughly says that if there is a degree one map $M\to N$, then $\F(M)\geq \F(N)$. (Here and later on, $\F$ will always serve as a placeholder for $\FR$ or $\FV$.) Moreover, they also satisfy an `extension axiom' which says that if one attaches finitely many cells of dimension less than the dimension of the manifold, then the value of the filling ratios does not change. The CW complex obtained in this way is called an `extension' of the original manifold.

If $M$ and $N$ have the same orientation behaviour, then there exists an extension $X$ of $N$ and a `degree one' map $M\to X$. (This is a special case of a more general theorem in \cite{Brunnbauer(2007a)}. In the orientable case this can be seen more directly by using the Hurewicz theorem.) Applying both axioms one sees that 
\[ \F(M)\geq \F(X) = \F(N). \]
Inverting the roles of $M$ and $N$ gives equality and proves the theorem.

The paper is organized as follows: in the next section necessary definitions are recalled. The main part of the proof can be found in section \ref{filling_axioms}, where $\FR$ and $\FV$ are shown to satisfy the comparison and extension axioms. In section \ref{constancy} the proof will be concluded and finally, the results and some open questions are discussed in the last section.

\begin{ackn*}
I would like to thank D.\,Kotschick for suggesting to investigate filling invariants and for many helpful discussions. I am also grateful to the anonymous referee for some useful remarks. Financial support from the \emph{Deutsche Forschungsgemeinschaft} is gratefully acknowledged.
\end{ackn*}


\section{Filling invariants}

The filling radius and the filling volume were introduced by Gromov in his Filling paper \cite{Gromov(1983)}. Using continuous piecewise smooth Riemannian metrics their definitions extend to simplicial complexes. Before recalling these definitions we will focus on the so-called `universal property' of the Banach space of all bounded functions on some set.

\begin{defn}
Let $f:Y\to X$ be a continuous map between metric spaces. The \emph{dilatation} of $f$ is given by
\[ \dil(f) := \sup_{y,y'\in Y, y\neq y'} \frac{d(f(y),f(y'))}{d(y,y')}, \]
i.\,e.\ it is the smallest Lipschitz constant for $f$. We define the \emph{dilatation of $f$ with respect to $y$} as 
\[ \dil(f,y) := \sup_{y'\in Y, y\neq y'} \frac{d(f(y),f(y'))}{d(y,y')}. \]
\end{defn}

For a set $V$ let $L^\infty(V)$ denote the Banach space of all bounded functions $f:V\to\mathbb{R}$ with the uniform norm $\|f\|:=\sup_{v\in V} |f(v)|$. It has the following \emph{universal property}:

\begin{lem}\label{univ_prop}
Let $Y\subset X$ be a nonempty subspace of a metric space and let $f:Y\to L^\infty(V)$ be a Lipschitz map. Then the map $F:X \to L^\infty(V)$ defined by
\[ F_x(v) := \inf_{y\in Y} (f_y(v) + \dil(f,y) \cdot d(x,y)) \]
is a Lipschitz continuous extension of $f$ with $\dil(F)=\dil(f)$ and $\dil(F,y)=\dil(f,y)$ for all $y\in Y$.
\end{lem}

The existence of an extension with the same dilatation as $f$ was shown by Gromov (see \cite{Gromov(1983)}, page 8). However, he used the extension 
\[ F'_x(v) := \inf_{y\in Y} (f_y(v) + \dil(f) \cdot d(x,y)), \]
which in general does not have the property $\dil(F,y)=\dil(f,y)$ for all $y\in Y$. Since in section \ref{filling_axioms} we will need that the dilatation with respect to points in $Y$ remains the same, we give a complete proof of Lemma \ref{univ_prop}.

\begin{proof}
A priori, $F$ is a map to the space of functions from $V$ to $[-\infty,\infty)$. 

First note that $F$ extends $f$, which can be seen as follows: for any $y\in Y$ one finds
\begin{align*}
0\leq f_y(v) - F_y(v) &= \sup_{y'\in Y} (f_y(v) - f_{y'}(v) - \dil(f,y')d(y,y')) \\
& \leq  \sup_{y'\in Y} (\dil(f,y')d(y,y')-\dil(f,y')d(y,y')) = 0
\end{align*}
because $\sup_{v\in V}|f_y(v)-f_{y'}(v)|= d(f_y,f_{y'}) \leq \dil(f,y') d(y,y')$. Hence $F_y\equiv f_y$.

Furthermore, one has
\begin{align*}
F_x(v)  &= \inf_{y\in Y} (f_y(v) + \dil(f,y) d(x,y)) \\
&\leq \inf_{y\in Y} (f_y(v) + \dil(f,y) d(x',y) + \dil(f,y) d(x,x')) \\
&\leq \inf_{y\in Y} (f_y(v) + \dil(f,y)d(x',y) + \dil(f) d(x,x')) \\
&= F_{x'}(v) + \dil(f) d(x,x')
\end{align*}
by the triangle inequality. This shows that $d(F_x,F_{x'})=\|F_x - F_{x'}\|\leq \dil(f)d(x,x')$. In particular $F_x$ is bounded for every $x\in X$ and $F$ is Lipschitz continous with $\dil(F)\leq\dil(f)$. The converse inequality between the dilatations is obvious since $F$ extends $f$.

Note that one always has $d(f_y,f_{y'}) \leq \dil(f,y)d(x,y) + \dil(f,y')d(x,y')$ for any $x\in X$. Hence
\begin{align*}
d(F_y,F_x) &= \sup_{v\in V} \left|\sup_{y'\in Y} (f_y(v) - f_{y'}(v) - \dil(f,y')d(x,y')) \right| \\
&\leq \dil(f,y)d(x,y),
\end{align*}
which proves the claim $\dil(F,y)=\dil(f,y)$.
\end{proof}

Now, assume $V$ to be a connected finite simplicial complex of dimension $n$. Let $\mathbb{K}$ denote $\mathbb{Z}$, $\mathbb{Q}$, or $\mathbb{Z}_2$.

\begin{defn}
Let $\iota:V\hookrightarrow X$ be a topological embedding into a metric space $X$. The \emph{$\mathbb{K}$-filling radius} of $\iota$ is defined as
\[ \FillRad_\mathbb{K}(\iota: V\hookrightarrow X) := \inf \{ r>0 | \iota_*: H_n(V;\mathbb{K})\to H_n(U_r(\iota V);\mathbb{K}) \mbox{ is zero}\}, \]
where $U_r(\iota V)$ denotes the $r$-neighborhood of the image $\iota V$ in $X$.
\end{defn}

To define the filling volume we need to specify a choice of volume for singular Lipschitz chains, i.\,e.\ singular chains whose simplices are Lipschitz continuous. Following Gromov (\cite{Gromov(1983)}, page 11 and section 4.1) we define the volume of a singular Lipschitz simplex $\sigma:\Delta^n\to X$ in a metric space $X$ by
\[ \Vol(\sigma) := \inf \{ \Vol(\Delta^n,g) | \sigma:(\Delta^n,g)\to X \mbox{ nonexpanding}\}, \]
where $g$ runs over all Riemannian metrics on $\Delta^n$. If $c=\sum_i r_i\sigma_i\in C_n(X;\mathbb{K})$ is a singular Lipschitz chain, its volume is defined as
\[ \Vol(c) := {\textstyle\sum_i}|r_i|\Vol(\sigma_i). \]
(In the case of $\mathbb{K}=\mathbb{Z}_2$ the `absolute value' $|r|$ is understood as zero for the trivial element $r=0$ and as one for $r\neq 0$.) For a Lipschitz cycle $z\in C_n (X;\mathbb{K})$ the \emph{$\mathbb{K}$-filling volume} is given by
\[ \FillVol_\mathbb{K}(z) := \inf_{\partial c=z} \Vol(c), \]
where the infimum is taken over all Lipschitz chains $c$ with boundary $\partial c=z$.

In case that $X$ is a Banach space, this definition corresponds to the `hyper-Euclidean volume' (see \cite{Gromov(1983)}, page 33). This choice of volume has the following nice property: let $(M,g)$ be a connected closed oriented Riemannian manifold of dimension $n\geq 2$ and let $W$ be an oriented $(n+1)$-dimensional manifold with boundary $\partial W=M$, for example $W=M\times [0,\infty)$. Consider Riemannian metrics $g'$ on $W$ such that their distance function restricted to the boundary is bounded from below by the distance function of $g$. Then the filling volume $\FillVol(M,g)$ defined below equals the infimum of the volumes $\Vol(W,g')$ of all such Riemannian metrics $g'$ on $W$ (\cite{Gromov(1983)}, Proposition 2.2.A). 

From now on, all singular simplices are assumed to be Lipschitz continuous.

The filling volume of an embedding $\iota:V\hookrightarrow X$ will be the filling volume of a canonical topdimensional homology class. If $V$ is a manifold, then the fundamental class provides such a class. In the case of simplicial complexes, we have to require the existence of a `fundamental class'.

\begin{defn}
A connected finite $n$-dimensional simplicial complex $V$ will be called \emph{$\mathbb{K}$-orientable} if $H_n(V;\mathbb{K})\cong\mathbb{K}$. As in the case of manifolds, $V$ is called \emph{$\mathbb{K}$-oriented} once a generator $[V]_\mathbb{K}\in H_n(V;\mathbb{K})$ is chosen. In case $\mathbb{K}=\mathbb{Q}$ we will always assume that $[V]_\mathbb{Q}$ lies in the integral lattice $H_n(V;\mathbb{Z})\subset H_n(V;\mathbb{Q})$.
\end{defn}

The last part of the definition simply says that for $\mathbb{K}=\mathbb{Q}$ one takes the integral fundamental class and allows filling cycles with rational coefficients.

\begin{defn}
Let $V$ be a connected finite $\mathbb{K}$-oriented simplicial complex. If $\iota:V\hookrightarrow X$ is a Lipschitz embedding into a metric space $X$, then one defines the \emph{$\mathbb{K}$-filling volume} of $\iota$ as
\[ \FillVol_\mathbb{K}(\iota:V\hookrightarrow X) := \FillVol_\mathbb{K}(\iota_*z) \]
where $z$ is a Lipschitz representative of $[V]_\mathbb{K}$.
\end{defn}

This is independent of the representing fundamental cycle. Namely, let $z$ and $z'$ be two Lipschitz representatives of $[V]_\mathbb{K}$. Then there exists a Lipschitz chain $b\in C_{n+1}(V;\mathbb{K})$ such that $z'=z+\partial b$ (with $n=\dim V$). Therefore it suffices to see that $\Vol(\iota_*b)=0$. But $\Vol(\iota_*b)\leq L^{n+1}\cdot\Vol(b)$ where $L>0$ is the Lipschitz constant of $\iota$ and $\Vol(b)=0$ since $b$ is a $(n+1)$-cycle in a $n$-dimensional complex.

For a compact metric space $(V,d)$ the \emph{Kuratowski embedding}
\begin{align*} 
\iota:V &\hookrightarrow L^\infty(V),\\ 
v &\mapsto d(v,\_)
\end{align*}
is an isometric embedding by the triangle inequality. 

Let $g$ be a continuous piecewise smooth Riemannian metric on a connected finite ($\mathbb{K}$-oriented) simplicial complex $V$. With the induced path metric, $V$ becomes a metric space. The associated Kuratowski embedding will be denoted by $\iota_g:V\hookrightarrow L^\infty(V)$. For this embedding the filling invariants are denoted by
\begin{align*}
\FillRad_\mathbb{K}(V,g) &:= \FillRad_\mathbb{K}(\iota_g) \;\;\;\mbox{and} \\
\FillVol_\mathbb{K}(V,g) &:= \FillVol_\mathbb{K}(\iota_g).
\end{align*}

\begin{defn}
We define the optimal filling ratios
\begin{align*}
\FR_\mathbb{K}(V) &:= \sup_g \frac{\FillRad_\mathbb{K}(V,g)^n}{\Vol(V,g)} \;\;\;\mbox{and} \\
\FV_\mathbb{K}(V) &:= \sup_g \frac{\FillVol_\mathbb{K}(V,g)^{n/(n+1)}}{\Vol(V,g)} .
\end{align*}
\end{defn}

These numbers are the smallest constants such that the filling inequalities
\begin{align*}
\FillRad_\mathbb{K}(V,g)^n &\leq A(V)\cdot\Vol(V,g) \;\;\;\mbox{and} \\
\FillVol_\mathbb{K}(V,g)^{n/(n+1)} &\leq B(V)\cdot\Vol(V,g)
\end{align*}
are satisfied for all Riemannian metrics $g$ on $V$. By Gromov's uiversal filling inequality one knows that for manifolds $M$ there are upper bounds for $\FR_\mathbb{K}(M)$ and $\FV_\mathbb{K}(M)$ depending only on the dimension. In particular, both filling ratios are finite for manifolds.


\section{Axioms for filling invariants}\label{filling_axioms}

The content of this section is the proof of certain properties of the filling ratios. We will show that $\FR$ and $\FV$ satisfy a comparison axiom and an extension axiom. Thereby a numerical invariant for simplicial complexes is said to fulfill a comparison axiom if any `degree one' map $V\to W$ implies an inequality between the values of this invariant for $V$ and for $W$. It has to be specified which maps are of `degree one' (e.\,g.\ $(n,1)$-monotone maps defined below) and often there are further assumptions on the maps, like surjectivity on fundamental groups or on some homology groups. (Examples can be found in \cite{Brunnbauer(2007a)} and below.)

An extension axiom is satisfied if attaching cells with dimension less than the dimension of the simplicial complex does not change the value of the considered invariant. Again, examples are provided by \cite{Brunnbauer(2007a)}.

The actual proof of Theorems \ref{main_thm_intro} and \ref{main_thm} uses only these two axioms and no other properties of the filling ratios. It will be given in section \ref{constancy}.

\begin{defn}
A simplicial map $f:V\to W$ between two $n$-dimensional simplicial complexes is called \emph{$(n,d)$-monotone} if the preimage of each open $n$-simplex of $W$ consists of at most $d$ open $n$-simplices in $V$.
\end{defn}

By a theorem of Hopf this is a generalization of the usual notion of degree for maps between manifolds (see \cite{Epstein(1966)}).

\begin{lem}[Comparison axiom for $\FR$ and $\FV$]\label{comp_axiom}
If $f:V\to W$ is an $(n,1)$-monotone map between connected finite ($\mathbb{K}$-oriented) simplicial complexes of dimension $n$ such that $f_*:H_n(V;\mathbb{K})\twoheadrightarrow H_n(W;\mathbb{K})$ is surjective, then
\[ \F_\mathbb{K}(V) \geq \F_\mathbb{K}(W) \]
for both $\F=\FR$ and $\F=\FV$.
\end{lem}

\begin{proof}
Let $g_2$ be a Riemannian metric on $W$. Choose a Riemannian metric $g_1$ on $V$ and set $g_1^t:=f^*g_2 + t^2 g_1$. This is a Riemannian metric on $V$. One may choose $t>0$ so small that
\[ \Vol(V,g_1^t) \leq \Vol(W,g_2)+\varepsilon \]
for a given $\varepsilon>0$. Denote the corresponding Kuratowski embeddings by $\iota_1^t$ and $\iota_2$. 

Since $f:(V,g_1^t)\to (W,g_2)$ is nonexpanding, there is a nonexpanding map $F:L^\infty(V)\to L^\infty(W)$ that extends $\iota_2\circ f$ by the universal property of $L^\infty(W)$. (Think of $V\subset L^\infty(V)$ via $\iota_1^t$.) Thus there is a commutative diagram
\[\xymatrix{
V \ar[r]^-f \ar@{^{(}->}[d]_-{\iota_1^t} & W \ar@{^{(}->}[d]^-{\iota_2} \\
L^\infty(V) \ar[r]^-F & L^\infty(W)
}\]
which gives for every $r>0$ 
\[\xymatrix{
H_n(V;\mathbb{K}) \ar@{->>}[r]^-{f_*} \ar[d]_-{(\iota_1^t)_*} & H_n (W,\mathbb{K}) \ar[d]^-{(\iota_2)_*} \\
H_n (U_r(\iota_1^t V);\mathbb{K}) \ar[r]^-{F_*} & H_n (U_r(\iota_2 W);\mathbb{K})
}\]
Therefore $\FillRad_\mathbb{K}(V,g_1^t) \geq \FillRad_\mathbb{K}(W,g_2)$ and this finally gives $\FR_\mathbb{K}(V)\geq \FR_\mathbb{K}(W)$.

Let $z\in C_n(V;\mathbb{K})$ represent $[V]_\mathbb{K}$. Then $f_*z$ represents $f_*[V]_\mathbb{K}=\pm [W]_\mathbb{K}$ (in case $\mathbb{K}=\mathbb{Q}$ look at the local degree) and one finds
\[ \FillVol_\mathbb{K}((\iota_2)_*(f_*z)) = \FillVol_\mathbb{K}(F_*((\iota_1^t)_* z)) \leq \FillVol_\mathbb{K}((\iota_1^t)_* z). \]
Hence $\FillVol_\mathbb{K}(W,g_2)\leq \FillVol_\mathbb{K}(V,g_1^t)$ and $\FV_\mathbb{K}(W)\leq \FV_\mathbb{K}(V)$.
\end{proof}

The proof of the extension axiom is more complicated. We will frequently use the following fact, which is a direct consequence of the universal property (Lemma \ref{univ_prop}).
\begin{cor}\label{no_matter}
If $i:(V,g)\hookrightarrow L^\infty(W)$ is an isometric embedding with $W$ any set, then
\begin{align*}
\FillRad_\mathbb{K}(i) &= \FillRad_\mathbb{K}(V,g) \;\;\;\mbox{and}\\
\FillVol_\mathbb{K}(i) &= \FillVol_\mathbb{K}(V,g).
\end{align*}
\end{cor}

We first investigate the extension axiom for the filling radius:

\begin{prop}[Extension axiom for $\FR$]\label{ext_FR}
Let $V'$ be an extension of $V$, i.\,e.\ $V'$ is obtained from $V$ by attaching finitely many cells of dimension less than $n:=\dim V$. Then
\[ \FR_\mathbb{K}(V') = \FR_\mathbb{K}(V). \]
\end{prop}

\begin{proof}
Since the inclusion $i:V\hookrightarrow V'$ is $(n,1)$-monotone and induces an isomorphism 
\[ i_*:H_n (V;\mathbb{K})\xrightarrow{\cong} H_n (V';\mathbb{K}), \]
the comparison axiom implies $\FR_\mathbb{K}(V) \geq \FR_\mathbb{K}(V')$.

For the converse inequality it suffices by induction to attach one $k$-cell at a time, $k<n$. Let $h:S^{k-1}\to V$ be the (simplicial) attaching map and let $g$ be a Riemannian metric on $V$. 

Consider all $R>0$ such that $h:(S^{k-1},g_R)\to (V,g)$ is nonexpanding, where $g_R$ denotes the round metric with radius $R$ on $S^{k-1}$. (In the case $k=1$ choose $R>0$ such that $23\pi R\geq d(h(-1),h(1))$.) Define Riemannian metrics $g'_R$ on $V'$ by thinking of $V'$ as
\[ V\;\;\;\cup_h \;\;\; (S^{k-1}\times[-1,0]) \;\;\;\cup\;\;\; (S^{k-1}\times [0,6]) \;\;\;\cup\;\;\; D^k \]
and taking
\[ g,\;\;\; (-sh^*g +(1+s)g_R)\oplus \pi R ds^2,\;\;\; g_R\oplus \pi R ds^2,\;\;\; g_R \]
on the respective parts. Here the last $g_R$ denotes the round metric of radius $R$ on the $k$-dimensional hemisphere. (For $k=1$ use $5\pi Rds^2$ on $S^0\times[-1,0]$.) Then the induced distance functions of $g$ and $g'_R$ coincide on $V$:
\[ d_{g'_R}|V \equiv d_g, \]
i.\,e.\ the inclusion $i:(V,g)\hookrightarrow (V',g'_R)$ is isometric in the strong sense. Hence by Corollary \ref{no_matter} we have
\[ \FillRad_\mathbb{K}(V,g) = \FillRad_\mathbb{K} (\iota_{g'_R}\circ i) \]
and therefore
\[ \FillRad_\mathbb{K}(V',g'_R)\leq\FillRad_\mathbb{K}(V,g) \]
since for any $r>0$ the $r$-neighborhood of $V'$ is larger than the one of $V$. Furthermore, note that the inclusion $((S^{k-1}\times[1,6])\cup D^k,g'_R)\subset (V',g'_R)$ is strongly isometric, too. We will write $V'_R:=\iota_{g'_R}V', V_R:=\iota_{g'_R}V$ and so on.

Next, we restrict our attention to radii $R>\FillRad_\mathbb{K}(V,g)$. This guarantees that the $r$-neighborhood of $((S^{k-1}\times[1,6])\cup D^k)_R$ does not meet the $r$-neighborhood of $V_R$ for any $r<R$. Thus, we may think of $U_r(((S^{k-1}\times[1,6])\cup D^k)_R)$ as some kind of `tubular neighborhood' and try to retract it to its core. This core is $k$-dimensional and plays therefore no role for $n$th homology. Hence, if $H_n(V;\mathbb{K})$ vanishes in $U_r(V'_R)$, then also in $U_r(V_R)$. We now concretize this idea.

Choose one of the radii $R$ with $R>\FillRad_\mathbb{K}(V,g)$ and, additionally, with $R>1$ as reference radius and call it $R_0$. Since the Kuratowski embedding $\iota_0 := \iota_{g'_{R_0}}$ is not differentiable on the attached $k$-cell, we need to choose a smooth approximation to get a tubular neighborhood. Therefore, let 
\[ \iota: (S^{k-1}\times [3+\delta/\pi R_0,6])\cup D^k \hookrightarrow L^\infty(V') \]
be a smooth embedding such that 
\[ d(\iota(x), \iota_0(x)) < \delta \]
for all $x\in (S^{k-1}\times [3+\delta/\pi R_0,6])\cup D^k$. Using the Kuratowski embedding $\iota_0$ on $V'\setminus ((S^{k-1}\times [3,6])\cup D^k)$ and linear interpolation on $S^{k-1}\times [3,3+\delta/\pi R_0]$, this defines a Lipschitz map $\iota:(V',g'_{R_0})\to L^\infty(V')$ which is $3\delta$-close to $\iota_0$. Denote by $K:=\dil(\iota)$ its Lipschitz constant and think of $\iota$ as a map $V'_{R_0}\to L^\infty(V')$.

With respect to $y\in V'_{R_0}\setminus ((S^{k-1}\times [2,6])\cup D^k)_{R_0}$ the dilatation $\dil(\iota,y)$ of $\iota$ is at most $1+\delta$. This holds because
\begin{align*}
\frac{d(\iota(y),\iota(y'))}{d(y,y')} &\leq \frac{d(y,y')+d(y',\iota(y'))}{d(y,y')} \\
&\leq 1+\frac{3\delta}{d(y,y')} \leq 1+\delta 
\end{align*}
for every $y'\in ((S^{k-1}\times [3,6])\cup D^k)_{R_0}$ and because $\iota$ is isometric on $V'_{R_0}\setminus ((S^{k-1}\times [3,6])\cup D^k)_{R_0}$.

Let $F:L^\infty(V')\to L^\infty(V')$ be an extension of $\iota:V'_{R_0}\to L^\infty(V')$ as in Lemma \ref{univ_prop}. Then
\[ F(U_r(V'_{R_0}\setminus ((S^{k-1}\times [2,6])\cup D^k)_{R_0})) \subset U_{r(1+\delta)}(V'_{R_0}\setminus ((S^{k-1}\times [2,6])\cup D^k)_{R_0}) \]
by the fact that $\dil(F,y)=\dil(\iota,y)\leq 1+\delta$ for all $y\in V'_{R_0}\setminus ((S^{k-1}\times [2,6])\cup D^k)_{R_0}$.

Denote $E:=\iota((S^{k-1}\times [3+\delta/\pi R_0,6])\cup D^k)$. Let $\nu(E)\to E$ be its normal bundle and let $\tau:\mathcal{O}\hookrightarrow L^\infty(V')$ be a tubular neighborhood for the trivial spray (i.\,e.\ the exponential map is given by vector addition) where $\mathcal{O}\subset \nu(E)$ is open, fiberwise starshaped with respect to the zero section (thus it allows a deformation retraction to the zero section), and fiberwise bounded by $\FillRad_\mathbb{K}(V,g)/2$. Furthermore, assume that
\[ \tau(\mathcal{O}|\iota((S^{k-1}\times[3{\textstyle\frac{1}{2}},6])\cup D^k)) \subset U_{r_0}(((S^{k-1}\times[2,6])\cup D^k)_{R_0}) \]
with $r_0=\FillRad_\mathbb{K}(V',g'_{R_0})$. By compactness there is an $\varepsilon>0$ such that 
\[ F(U_\varepsilon(((S^{k-1}\times[4,6])\cup D^k)_{R_0})) \subset \tau(\mathcal{O}|\iota((S^{k-1}\times[3{\textstyle\frac{1}{2}},6])\cup D^k)). \]
Moreover, choosing $\varepsilon < (r_0-3\delta)/K$ guarantees that
\[ F(U_\varepsilon(((S^{k-1}\times [2,6])\cup D^k)_{R_0})) \subset U_{r_0}(V'_{R_0}). \]

\begin{claim*}
There is an $R\geq R_0$ such that $H_n(V;\mathbb{K})$ vanishes for all $r>\FillRad_\mathbb{K}(V',g'_R)$ in
\[ U_r(V'_{R_0}\setminus ((S^{k-1}\times[2,6])\cup D^k)_{R_0}) \;\cup\; U_\varepsilon (((S^{k-1}\times[2,6])\cup D^k)_{R_0}). \]
\end{claim*}

\begin{proof}[Proof of the claim]
Choose a positive real number $C\leq 1$ with $C\pi\FillRad_\mathbb{K}(V,g)<\varepsilon$ and choose 
\[ R \geq \max( R_0/C,\diam(V'_{R_0})/C\pi). \] 

The identity on $V'$ gives a nonexpanding homeomorphism $f:V'_R\to V'_{R_0}$. Let $\tilde F:L^\infty(V')\to L^\infty(V')$ be an extension of $f$ as in Lemma \ref{univ_prop}. Then for any $\FillRad_\mathbb{K}(V',g'_R)< r \leq\pi\FillRad_\mathbb{K}(V,g)$ one finds
\[ \tilde F(U_r(((S^{k-1}\times [2,6])\cup D^k)_R)) \subset U_\varepsilon(((S^{k-1}\times [2,6])\cup D^k)_{R_0}) \]
since $\dil(\tilde F,y)=\dil(f,y)\leq C$ for any $y\in ((S^{k-1}\times [2,6])\cup D^k)_R$. This follows from
\[ \frac{d(f(y),f(y'))}{d(y,y')} \leq C \]
which holds for all $y'\in V'_R$ since on $((S^{k-1}\times[1,6])\cup D^k)_R$ the map $f$ is the contraction by the factor $R_0/R \leq C$ and for the other $y'$ note that the numerator is bounded from above by $\diam(V'_{R_0})$ and the denominator is bounded from below by $\pi R$. 

Hence $\tilde F$ maps $U_r(V'_R)$ to 
\[ U_r(V'_{R_0}\setminus ((S^{k-1}\times[2,6])\cup D^k)_{R_0}) \cup U_\varepsilon (((S^{k-1}\times[2,6])\cup D^k)_{R_0}). \]
Therefore $H_n(V;\mathbb{K})$ vanishes therein and the claim is proved.
\end{proof}

Note that $\FillRad_\mathbb{K}(V',g'_R)\geq \FillRad_\mathbb{K}(V',g'_{R_0})=r_0$ by the universal property. Therefore, applying $F$ shows that $H_n(V;\mathbb{K})$ also vanishes in
\[ U_{r(1+\delta)}(V'_{R_0}\setminus ((S^{k-1}\times [4,6])\cup D^k)_{R_0}) \cup \tau(\mathcal{O}|\iota((S^{k-1}\times[3{\textstyle\frac{1}{2}},6])\cup D^k)) \]
for all $\delta>0$ and $r>\FillRad_\mathbb{K}(V',g'_R)$. Using the tubular neighborhood retraction (this is where the fiberwise bound on $\mathcal{O}$ comes in) and a Mayer-Vietoris argument one sees that $H_n(V;\mathbb{K})$ maps to zero in
\[ U_{r(1+\delta)}(V'_{R_0}\setminus \mathring{D}^k_{R_0}). \]
Since this holds for all $\delta>0$ and $r>\FillRad_\mathbb{K}(V',g'_R)$ it follows that
\[ \FillRad_\mathbb{K}(V'\setminus\mathring{D}^k,g'_{R_0}) \leq \FillRad_\mathbb{K}(V',g'_R). \]

The retraction $(V'\setminus \mathring{D}^k,g'_{R_0})\to (V,g)$ is nonexpanding by the choice of $g'_{R_0}$, therefore
\[ \FillRad_\mathbb{K}(V,g) \leq \FillRad_\mathbb{K}(V'\setminus \mathring{D}^k,g'_{R_0}) \leq \FillRad_\mathbb{K}(V',g'_R). \]
Note that we have actually proved that $\FillRad_\mathbb{K}(V,g)=\FillRad_\mathbb{K}(V',g'_R)$.

Since $\Vol(V,g)=\Vol(V',g'_R)$ one gets $\FR_\mathbb{K}(V)\leq \FR_\mathbb{K}(V')$. This finishes the proof of Proposition \ref{ext_FR}.
\end{proof}

To finish this section, we will proof the extension axiom for $\FV$. Note that an extension of a $\mathbb{K}$-oriented simplicial complex is again $\mathbb{K}$-oriented and has the same fundamental class.

\begin{lem}[Extension axiom for $\FV$]\label{ext_FV}
Let $V'$ be an extension of $V$. Then
\[ \FV_\mathbb{K} (V') = \FV_\mathbb{K}(V). \]
\end{lem}

\begin{proof}
Since the inclusion $i:V\hookrightarrow V'$ is $(n,1)$-monotone and induces an isomorphism in the $n$-th homology, the inequality $\FV_\mathbb{K} (V) \geq \FV_\mathbb{K}(V')$ holds by the comparison axiom. 

Let $g$ be a Riemannian metric on $V$ and extend it over $V'$ such that the inclusion $i:V\hookrightarrow V'$ is isometric in the strong sense (see the proof of Proposition \ref{ext_FR}). Call this Riemannian metric $g'$. Then $\iota_{g'}\circ i:V\hookrightarrow L^\infty(V')$ is an isometric embedding and $\FillVol_\mathbb{K}(V,g) = \FillVol_\mathbb{K}(\iota_{g'}\circ i)$ by Corollary \ref{no_matter}. Choose a Lipschitz chain $z\in C_n(V;\mathbb{K})$ that represents $[V]_\mathbb{K}$. Then $i_*z$ is a Lipschitz chain that represents $[V']_\mathbb{K}$ and 
\[ \FillVol_\mathbb{K}(V,g) = \FillVol_\mathbb{K}(\iota_{g'*}(i_*z)) = \FillVol_\mathbb{K}(V',g'). \]
Since $\Vol(V,g)=\Vol(V',g')$ it follows that $\FV_\mathbb{K}(V)\leq \FV_\mathbb{K}(V')$. 
\end{proof}


\section{Constancy of the filling invariants}\label{constancy}

In this section we will prove our main theorem, which states that $\FR$ and $\FV$ depend only on dimension and orientability. 

\begin{thm}\label{main_thm}
Let $M$ and $N$ be two connected closed manifolds of dimension $n\geq 3$. If either both are orientable or both are nonorientable, then
\[ \F_\mathbb{K}(M) = \F_\mathbb{K}(N) \]
for both $\F=\FR$ and $\F=\FV$. If only $M$ is orientable, then
\[ \F_{\mathbb{Z}_2}(M) \leq \F_{\mathbb{Z}_2}(N). \]
\end{thm}

This theorem obviously includes Theorem \ref{main_thm_intro}. The proof relies on the following observation.

\begin{thm}\label{mon_map_thm}
Let $M$ and $N$ be two connected closed manifolds of dimension $n\geq 3$. If either both are orientable or $N$ is nonorientable, then there exists an extension $V$ of $M$ and an $(n,1)$-monotone map $h:N\to V$ with $h_*[N]_\mathbb{K}=i_*[M]_\mathbb{K}=[V]_\mathbb{K}$, where $i:M\hookrightarrow V$ is the inclusion.
\end{thm}
 
This is \cite{Brunnbauer(2007a)}, Theorem 7.4 for the case of the trivial group $\pi=1$. We will not give a complete proof but we will sketch the (easier) case in which $M$ is orientable.

In a first step, one has to find an extension $V$ of $M$ and a map $h':N\to V$ with $h_*[N]_\mathbb{K}=i_*[M]_\mathbb{K}$. To construct this extension one attaches $2$-cells to $M$ to kill the fundamental group, then one attaches $3$-cells to kill the second homotopy group and so on. The complex obtained after attaching $k$-cells will be called $M(k)$. Then the complex $M(n)$ and the pair $(M(n),M(n-1))$ are both $(n-1)$-connected. Since $M(n-1)$ is simply-connected (because $n\geq 3$), the two vertical Hurewicz homomorphisms on the right-hand side of the following diagram are isomorphisms.
\[\xymatrix{
& \pi_n(M(n-1)) \ar[r] \ar[d] & \pi_n(M(n)) \ar[r] \ar[d]^-\cong & \pi_n(M(n),M(n-1)) \ar[d]^-\cong \\
0 \ar[r] & H_n(M(n-1);\mathbb{Z}) \ar[r] & H_n(M(n);\mathbb{Z}) \ar[r] & H_n(M(n),M(n-1);\mathbb{Z})
}\]

A diagram chase shows that the first vertical Hurewicz homomorphism $\pi_n(M(n-1))\to H_n(M(n-1);\mathbb{Z})$ is surjective. Therefore, there is a map $s:S^n\to M(n-1)$ such that $s_*[S^n]_\mathbb{K}=i_*[M]_\mathbb{K}$. Let $M\subset V\subset M(n-1)$ be a finite subcomplex such that $s_*[S^n]_\mathbb{K}=i_*[M]_\mathbb{K}$ holds in $V$. Note that $V$ is an extension of $M$. Choose a degree one map $N\to S^n$ and compose it with $s$ to get the desired map $h':N\to V$ with $h'_*[N]_\mathbb{K}=i_*[M]_\mathbb{K}$.

In the second step this map $h'$ is deformed in such a way that the local degree of any point inside a top-dimensional cell of $V$ becomes one. The resulting map $h:N\to V$ is then $(n,1)$-monotone. The details of this step may be found in \cite{Brunnbauer(2007a)}, Lemma 2.2. The used deformation techniques go back to Hopf and were also applied in \cite{Epstein(1966)}.

The argument for the case of nonorientable $M$ is similar but more involved. This is due to the fact that in this case $H_n(M(n-1);\mathbb{Z})=0$ and that therefore the above reasoning does not work. The reader is referred to \cite{Brunnbauer(2007a)}, Theorem 7.3.

\begin{proof}[Proof of Theorem \ref{main_thm}]
By Theorem \ref{mon_map_thm} there is an extension $V$ of $M$ and an $(n,1)$-monotone map $h:N\to V$ with $h_*[N]_\mathbb{K}=i_*[M]_\mathbb{K}=[V]_\mathbb{K}$, where $i:M\hookrightarrow V$ is the inclusion. Then 
\[ \F_\mathbb{K}(N) \geq \F_\mathbb{K}(V) = \F_\mathbb{K}(M) \]
by the comparison axiom and the extension axiom. Changing the roles of $M$ and $N$ gives equality in the cases where $M$ and $N$ are both orientable or both nonorientable.
\end{proof}

Note that the proof used only the axioms and no other properties of the filling ratios. Thus it works for all numerical invariants of simplicial complexes that satisfy the comparison and extension axiom.


\section{Remarks and questions}

We have eight distinguished positive numbers in each dimension $n\geq 3$ that satisfy the following inequalities:
\begin{align*}
\FR^{or}_\mathbb{Q}(n) \leq \FR^{or}_\mathbb{Z}(n) &\geq \FR^{or}_{\mathbb{Z}_2}(n) \leq \FR^{non\mbox{-}or}_{\mathbb{Z}_2}(n) \;\;\mbox{and} \\
\FV^{or}_\mathbb{Q}(n) \leq \FV^{or}_\mathbb{Z}(n) &\geq \FV^{or}_{\mathbb{Z}_2}(n) \leq \FV^{non\mbox{-}or}_{\mathbb{Z}_2}(n).
\end{align*}
The first two inequalities of each line are direct consequences of the definitions, the last inequalities stem from Theorem \ref{main_thm}. We do not know about strict inequalities or equalities.

The exact values of these constants are also not known. Actually, the filling radius is only known in three cases:
\begin{align*}
\FillRad_\mathbb{K}(\mathbb{R}P^n, g_0) &= \pi/6 ,\\
\FillRad_\mathbb{K}(S^n,g_0) &= {\textstyle\frac{1}{2}}\arccos(-{\textstyle\frac{1}{n+1}}), \\
\FillRad_\mathbb{Z}(\mathbb{C}P^2,g_0) &= {\textstyle\frac{1}{2}}\arccos(-{\textstyle\frac{1}{3}}), \\ 
\FillRad_\mathbb{Q} (\mathbb{C}P^k,g_0) &= {\textstyle\frac{1}{2}}\arccos(-{\textstyle\frac{1}{3}}),
\end{align*}
where $g_0$ denotes the `round' metrics of constant curvature one respectively the Fubini-Study metric and $\mathbb{K}$ stands for all possible coefficient rings chosen from $\mathbb{Z}$, $\mathbb{Z}_2$, and $\mathbb{Q}$. (See \cite{Katz(1983)} and \cite{Katz(1991b)}.) By computing the ratio $\FillRad_\mathbb{K}(M,g)^n/\Vol(M,g)$ for these examples, it follows that the round projective space is not maximizing $\FR$ in dimensions $n\neq 1$ and that the standard complex space is not maximizing in even dimensions $n\geq 4$. By this calculations, one is tempted to conjecture that the supremum that defines $\FR$ is a maximum and that the round metric on the sphere maximizes this ratio. 

For the filling volume the situation is far more vague: one does not know the exact value for a single Riemannian manifold, not even for the circle.

In the case of surfaces the comparison axiom has the following consequence:

\begin{cor}\label{two_dim}
Denote the connected closed surface of genus $g$ by $\Sigma_g$ in the orientable case and by $N_g$ in the nonorientable case. Then
\[ \F_\mathbb{K} (\Sigma_g) \leq \F_\mathbb{K} (\Sigma_{g+1}) \]
and 
\[ \F_{\mathbb{Z}_2} (N_g) \leq \F_{\mathbb{Z}_2} (N_{g+1}). \]
Moreover, since $N_{2g+1}\cong \Sigma_g \# \mathbb{R}P^2$ the inequality
\[ \F_{\mathbb{Z}_2}(\Sigma_g) \leq \F_{\mathbb{Z}_2} (N_{2g+1}) \]
holds.
\end{cor}

It would be interesting to know whether equality always holds or whether strict inequality can indeed occur in the first two inequalities of the corollary. The following example may give an idea what can go wrong in dimension $n=2$.

We already mentioned the following fact: let $M$ be a connected closed oriented manifold of dimension at least two and let $d$ be a metric on $M$ (not necessarily steming from a Riemannian metric). Let $W$ be another oriented manifold with boundary $\partial W= M$. Then
\[ \FillVol_\mathbb{Z}(M,d) = \inf\{\Vol(W,g')\;\; |\;\; d_{g'}|_M \geq d \}. \]

This was proved in \cite{Gromov(1983)}, Proposition 2.2.A. Note that this infimum does not depend on the topology of $W$. (Indeed, one can always take $W=M\times[0,\infty)$.) So this theorem resembles the main theorem of this article and actually can be (and is) proved by similar methods. But in paragraph 2.2.B (2) of \cite{Gromov(1983)} it is shown that this Theorem does not hold if the dimension of $W$ is two. Thus by analogy, it may well be that strict inequalities occur in Corollary \ref{two_dim}.

However, note that Gromov's counterexample is not Riemannian in the sense that the metric $d$ on the circle $\partial W=S^1$ is not geodesic.


\bibliographystyle{amsalpha}
\bibliography{asymp_invar.bib}

\end{document}